\documentclass[12pt]{amsart}

\usepackage{a4wide}
\usepackage{amsmath, amssymb}
\usepackage{amsthm}
\usepackage[utf8]{inputenc}
\usepackage{subfigure}
\usepackage{enumerate}
\usepackage{cancel}
\usepackage{xspace}
\usepackage{float}

\usepackage{mathrsfs}
\usepackage[colorlinks, citecolor=blue, anchorcolor=blue, linkcolor=blue, urlcolor=blue]{hyperref}
\usepackage{color,xcolor}

\usepackage{wrapfig}

\usepackage{tikz}
\usetikzlibrary{patterns}

\usepackage{algorithm}
\usepackage{algorithmic}

\def\SYoung#1{\vbox{\smallskip\offinterlineskip
    \halign{&\vbox{##}\kern-\SThickness\cr #1}}}

\newdimen\SSquaresize \SSquaresize=4.5pt
\newdimen\SThickness \SThickness=.15pt
\newdimen\SCorrection \SCorrection=7pt

\def\SCarre#1{\hbox{\vrule width \SThickness
   \vbox to \SSquaresize{\hrule height \SThickness\vss
      \hbox to \SSquaresize{\hss$\scriptstyle#1$\hss}
   \vss\hrule height\SThickness}
   \unskip\vrule width \SThickness}
   \kern-\SThickness}

\makeatletter \@addtoreset{equation}{section}

\newtheorem{theorem}{Theorem}[section]

\newtheorem{lemma}[theorem]{Lemma}

\newtheorem{remark}[theorem]{Remark}

\title[A simple proof of higher order Tur\'{a}n inequalities]{A simple proof of higher order Tur\'{a}n inequalities for Boros-Moll sequences}

\author[James J. Y. Zhao]{James Jing Yu Zhao}
       \address{School of Accounting, Guangzhou College of Technology and Business,
       Foshan 528138, People's Republic of China.}
       \email{zhao@gzgs.edu.cn}

\keywords{Log-concavity; higher order Tur\'{a}n inequalities; Boros-Moll sequences}

\begin{document}

\begin{abstract}
Recently, the higher order Tur\'{a}n inequalities for the Boros-Moll sequences $\{d_\ell(m)\}_{\ell=0}^m$ were obtained by Guo. In this paper, we show a different approach to this result. Our proof is based on a criterion derived by Hou and Li, which need only checking four simple inequalities related to sufficiently sharp bounds for $d_\ell(m)^2/(d_{\ell-1}(m)d_{\ell+1}(m))$. In order to do so, we adopt the upper bound given by Chen and Gu in studying the reverse ultra log-concavity of Boros-Moll polynomials, and establish a desired lower bound for $d_\ell(m)^2/(d_{\ell-1}(m)d_{\ell+1}(m))$ which also implies the log-concavity of $\{\ell! d_\ell(m)\}_{\ell=0}^m$ for $m\geq 2$.
We also show a sharper lower bound for $d_\ell(m)^2/(d_{\ell-1}(m)d_{\ell+1}(m))$ which may be available for some deep results on inequalities of Boros-Moll sequences.
\end{abstract}

\subjclass{Primary 05A20; 11B83}

\thanks{This work
was partially supported by the National Natural Science Foundation of China Grant No. 11971203.}

\maketitle

\section{Introduction}

This paper is concerned with the higher order Tur\'{a}n inequalities for the Boros-Moll sequences $\{d_\ell(m)\}_{\ell=0}^m$. The term $d_\ell(m)$ is the coefficient of $x^\ell$ in the Boros-Moll polynomials
\begin{align}\label{eq:B-M-poly1}
P_m(x)
=\sum_{j,k} \binom{2m+1}{2j}\binom{m-j}{k}\binom{2k+2j}{k+j}
 \frac{(x+1)^j(x-1)^k}{2^{3(k+j)}},
\end{align}
which arise in the following evaluation of a quartic integral
$$
\int_0^\infty\frac{1}{(t^4+2xt^2+1)^{m+1}} dt
=\frac{\pi}{2^{m+3/2}(x+1)^{m+1/2}} P_m(x)
$$
for $x>-1$ and $m\in\mathbb{N}$.
Using Ramanujan's Master Theorem, \eqref{eq:B-M-poly1} can be restated as
\begin{align*}
P_m(x)=2^{-2m}\sum_{k=0}^m 2^k \binom{2m-2k}{m-k}\binom{m+k}{k}(x+1)^k.
\end{align*}
So, the formula of $d_\ell(m)$ is given by
\begin{align}\label{eq:B-M-seq}
d_\ell(m)=2^{-2m}\sum_{k=\ell}^m 2^k \binom{2m-2k}{m-k}\binom{m+k}{k}\binom{k}{\ell},
\end{align}
for $m\geq \ell \geq 0$,
see \cite{Boros-Moll-1999, Boros-Moll-2001, Boros-Moll-2004, Moll2002}.

A sequence $\{a_n\}_{n\geq 0}$ with real numbers is said to satisfy the \emph{Tur\'{a}n inequalities} or to be \emph{log-concave} if
\begin{align}\label{eq:df-log-concave}
a_n^2-a_{n-1}a_{n+1}\geq 0
\end{align}
for any $n\geq 1$.
The Tur\'{a}n inequalities \eqref{eq:df-log-concave} are also called the Newton's inequalities \cite{Craven-Csordas1989, CNV1986, Niculescu2000}.
A polynomial is said to be log-concave if the sequence of its coefficients is log-concave.

Boros and Moll \cite{Boros-Moll-2004} introduced the notion of infinite log-concavity and conjectured that the sequence $\{d_\ell(m)\}_{\ell=0}^m$ is infinitely log-concave, which received considerable attention. A number of interesting results were obtained.
Moll \cite{Moll2002} posed a conjecture that $\{d_\ell(m)\}_{\ell=0}^m$ is log-concave, which was later proved by Kauers and Paule \cite{Kauers-Paule} with a symbolic method called \emph{computer algebra}.
Chen, Dou, and Yang \cite{Chen-Dou-Yang} proved two conjectures of Br\"{a}nd\'{e}n \cite{Branden} concerning the Boros-Moll polynomials, and hence obtained $2$-log-concavity and $3$-log-concavity of $P_m(x)$. Chen and Xia \cite{Chen-Xia} also showed that $P_m(x)$ are ratio monotone which implies the log-concavity and the spiral property. Moreover, Chen and Gu \cite{Chen-Gu2009} proved the reverse ultra log-concavity of $P_m(x)$. See \cite{CWX2011, Chen-Xia-2} for more results on Boros-Moll polynomials.

A real sequence $\{a_n\}_{n\geq 0}$ is said to satisfy the \emph{higher order Tur\'{a}n inequalities} or cubic Newton inequalities if for all $n\geq 1$,
\begin{align}\label{eq:ho-Turan}
 4(a_n^2-a_{n-1}a_{n+1})(a_{n+1}^2-a_n a_{n+2})
-(a_{n} a_{n+1}-a_{n-1} a_{n+2})^2 \geq 0,
\end{align}
see \cite{Dimitrov1998, Niculescu2000, Rosset1989}.
The Tur\'{a}n inequalities and the higher order Tur\'{a}n inequalities are related to the Laguerre-P\'{o}lya class of real entire functions, see \cite{Dimitrov1998, Szego1948}.
A real entire function $$\psi(x)=\sum_{n=0}^{\infty} a_n \frac{x^n}{n!}$$ is said to belong to the Laguerre-P\'{o}lya class, denoted by $\psi \in \mathcal{L}$-$\mathcal{P}$, if
$$\psi(x)=c x^m e^{-\alpha x^2+\beta x}\prod_{k=1}^{\infty}(1+x/x_k)e^{-x/x_k},$$ where $c,\beta,x_k$ are real, $\alpha\geq 0$, $m$ is a nonnegative integer and $\sum x_k^{-2}<\infty$.
Jensen \cite{Jensen} proved that a real entire function $\psi(x)\in \mathcal{L}$-$\mathcal{P}$ if and only if for any integer $n>0$, the $n$-th associated Jenson polynomial
$$
J_n(x)=\sum_{k=0}^n \binom{n}{k} a_k x^k
$$
are hyperbolic, i.e., $J_n(x)$ has only real zeros. 
This result was also obtained by P\'{o}lya and Schur \cite{Polya-Schur1933}.
Besides, if a real entire function $\psi \in \mathcal{L}$-$\mathcal{P}$, then its Maclaurin coefficients satisfy \eqref{eq:df-log-concave}, see \cite{Craven-Csordas1989, Csordas-Varga1990}.
Moreover, Dimitrov \cite{Dimitrov1998} proved that the higher order Tur\'{a}n inequalities \eqref{eq:ho-Turan}, an extension of \eqref{eq:df-log-concave}, is also a necessary condition for  $\psi \in \mathcal{L}{\textrm{-}}\mathcal{P}$.

The $\mathcal{L}$-$\mathcal{P}$ class has close relation with the Riemann hypothesis. Let $\zeta$ and $\Gamma$ denote the Riemann zeta function and the gamma function, respectively. The Riemann $\xi$-function is defined by
\begin{equation*}
\xi(iz)=\frac{1}{2}\left(z^2-\frac{1}{4}\right)\pi^{-\frac{z}{2}-\frac{1}{4}}
  \Gamma\left(\frac{z}{2}+\frac{1}{4}\right)
  \zeta\left(z+\frac{1}{2}\right),
\end{equation*}
see Boas \cite{Boas1954}.
It is well known that the Riemann $\xi$-function is an entire function of order $1$ and can be restated as
\begin{equation}\label{eq:R-xi-r}
 \frac{1}{8}\,\xi\!\left(\frac{x}{2}\right)
=\sum\limits_{k=0}^{\infty}(-1)^k\, \hat{b}_k\, \frac{x^{2k}}{(2k)!},
\end{equation}
where
\begin{align*}
\hat{b}_k=\int_{0}^{\infty} t^{2t} \Phi(t) dt
\quad \textrm{and} \quad
\Phi(t)=\sum_{n=0}^\infty (2n^4 \pi^2 e^{9t} - 3n^2 \pi e^{5t}) \exp(-n^2 \pi e^{4t}),
\end{align*}
see P\'{o}lya \cite{Polya1927}. Set $z=-x^2$ in \eqref{eq:R-xi-r}. Then one obtain an entire function of order $1/2$, denoted by $\xi_1(z)$, that is,
\begin{equation*}
\xi_1(z)=\sum_{k=0}^\infty \frac{k!}{(2k)!}\hat{b}_k \frac{z^k}{k!}.
\end{equation*}
So the Riemann hypothesis holds if and only if $\xi_1(z)\in \mathcal{L}$-$\mathcal{P}$. See \cite{CJW2019, CNV1986, Dimitrov1998} for more details.

For the deep relation stated above, the higher order Tur\'{a}n inequalities have  received great attention. Many interesting sequences were showed to satisfy the higher order Tur\'{a}n inequalities \eqref{eq:ho-Turan}. For instance, the Riemann $\xi$-function was proved to satisfy \eqref{eq:ho-Turan} by Dimitrov and Lucas \cite{DL2011}.
By using the Hardy-Ramanujan-Rademacher formula, Chen, Jia, and Wang \cite{CJW2019} proved that the partition function $p(n)$ satisfies the higher order Tur\'{a}n inequalities for $n\geq 95$ and the $3$-rd associated Jensen polynomials $\sum_{k=0}^3 \binom{3}{k} p(n+k)x^k$ have three distinct zeros, and hence confirm a conjecture of Chen \cite{Chen2017}.

Griffin, Ono, Rolen, and Zagier \cite{GORZ2019} showed that Jensen polynomials for a large family of functions, including those associated to $\zeta(s)$ and the partition function $p(n)$, are hyperbolic for sufficiently large $n$. This result gave evidence for Riemann hypothesis.
Wang \cite{Wang2019} gave a sufficient condition to \eqref{eq:ho-Turan} and obtained some higher order Tur\'{a}n inequalities for sequences $\{S_n/n!\}_{n\geq 0}$, where $S_n$ are the Motzkin numbers, the Fine numbers, the Franel numbers of order $3$ and the Domb numbers. Hou and Li \cite{Hou-Li2021} presented a different sufficient condition on the higher order Tur\'{a}n inequalities for $n\geq N$ and also developed an algorithm to find the number $N$ for certain kind of sequences.

Recently, Guo \cite{Guo} proved the following result by showing an equivalent form of the higher order Tur\'{a}n inequalities.
\begin{theorem}[Guo]\label{thm:ht-B-M-ell}
For $m\geq 3$ and $1\leq \ell \leq m-2$, the sequence $\{d_\ell(m)\}_{\ell=0}^m$ satisfies the higher order Tur\'{a}n inequalities. That is,
\begin{align}\label{eq:ht-B-M-ell}
& 4(d_\ell(m)^2-d_{\ell-1}(m)d_{\ell+1}(m))
 (d_{\ell+1}(m)^2-d_\ell(m)d_{\ell+2}(m)) \\
&\ \geq
 (d_\ell(m)d_{\ell+1}(m)-d_{\ell-1}(m)d_{\ell+2}(m))^2.\nonumber
\end{align}
\end{theorem}

The objective of this paper is to give a new proof of Theorem \ref{thm:ht-B-M-ell}. More precisely, we show a proof of a slightly sharp version of \eqref{eq:ht-B-M-ell}.
\begin{theorem}\label{thm:ht-B-M-ell-sharp}
For $m\geq 3$ and $1\leq \ell \leq m-2$, the higher order Tur\'{a}n inequalities in \eqref{eq:ht-B-M-ell} hold strictly.
\end{theorem}

We shall prove Theorem \ref{thm:ht-B-M-ell-sharp} by using a sufficient condition on the higher order Tur\'{a}n inequalities given by Hou and Li \cite{Hou-Li2021}. To apply this criterion, we need only to check four simple inequalities associated with a set of sufficiently sharp bounds for $d_\ell(m)^2/(d_{\ell-1}(m)d_{\ell+1}(m))$. Notice that Hou and Li provided an algorithm to find the bounds which is very efficient for sequences with one parameter. But, their algorithm does not work for this case since $d_\ell(m)$ has two parameters $\ell$ and $m$.
So, we establish a desired lower bound by hand in Sections \ref{Sec:2} and \ref{Sec:3}, and adopt the upper bound obtained by Chen and Gu \cite{Chen-Gu2009}. Then we complete the proof of Theorem \ref{thm:ht-B-M-ell-sharp} in Section \ref{Sec:ht-ell}.

\section{A sharper lower bound for $d_\ell(m+1)/d_\ell(m)$}\label{Sec:2}

The aim of this section is to prove a sharp enough lower bound for $d_\ell(m+1)/d_\ell(m)$, so that it will lead to a sufficiently sharp lower bound for $d_\ell(m)^2/(d_{\ell-1}(m)d_{\ell+1}(m))$, which satisfies the requirement of our proof of Theorem \ref{thm:ht-B-M-ell-sharp}.
For $m\geq 1$ and $0\leq \ell \leq m$, set
\begin{align}\label{eq:def-Lml}
L(m,\ell)=\frac{4m^2+7m-2\ell^2+3}{2(m+1)(m-\ell+1)}
  +\frac{\ell\sqrt{4\ell^4+8\ell^2m+5\ell^2+m}}{2(m+1)(m-\ell+1)\sqrt{m+\ell^2}}.
\end{align}

The main result of this section is as follows.
\begin{theorem}\label{thm:lb-ratdlm1m}
Let $L(m,\ell)$ be given in \eqref{eq:def-Lml}.
For any $m\geq 2$ and $1\leq \ell \leq m-1$,we have
\begin{align}\label{ineq:lb-dlm1m}
\frac{d_\ell(m+1)}{d_\ell(m)}>L(m,\ell),
\end{align}
and for $m\geq 1$, we have
\begin{align}\label{eq:dlmlb0m}
\frac{d_0(m+1)}{d_0(m)}=L(m,0),\qquad 
\frac{d_m(m+1)}{d_m(m)}=L(m,m). 
\end{align}
\end{theorem}

Kauers and Paule \cite{Kauers-Paule} used a computer algebra system to build the following recurrence relations for $d_\ell(m)$, which will be adopted in our proofs.
\begin{theorem}[Kauers-Paule]\label{Thm:rec-dlm-m}
For $m\geq 1$ and $0\leq \ell \leq m+1$, there holds
\begin{align}
&\ 4(m+1)(m+2)(m+2-\ell)d_\ell(m+2)\label{eq:rec-dlm-m}\\
=&\ 2(m+1)(8m^2+24m-4\ell^2+19)d_\ell(m+1)
  -(4m+3)(4m+5)(m+\ell+1)d_\ell(m).\nonumber
\end{align}\vskip -23pt
\begin{align}
2(m+1)d_\ell(m+1)=&\ 2(m+\ell)d_{\ell-1}(m)+(4m+2\ell+3)d_\ell(m),\label{eq:reclm1}\\
2(m+1)(m+1-\ell)d_\ell(m+1)=&\ (4m-2\ell+3)(m+\ell+1)d_{\ell}(m)\label{eq:reclm2}\\
&\ -2\ell(\ell+1)d_{\ell+1}(m).\nonumber
\end{align}
\end{theorem}
Note that Moll \cite{Moll2007} also independently found the relation \eqref{eq:rec-dlm-m} via the WZ-method \cite{PWZ, WZ}.
Based on these bounds, Kauers and Paule \cite{Kauers-Paule} obtained the following bound for $d_\ell(m+1)/d_\ell(m)$, that is,
\begin{align}\label{ineq:KP-lb}
\frac{d_\ell(m+1)}{d_\ell(m)}
\geq\frac{4m^2+7m+\ell+3}{2(m+1)(m-\ell+1)},\quad 0\leq \ell \leq m.
\end{align}
To derive the strict ratio monotonicity of $P_m(x)$, Chen and Xia \cite{Chen-Xia} showed a slightly sharp version of \eqref{ineq:KP-lb}, that is,
\begin{align}\label{ineq:KP-lb-sharp}
\frac{d_\ell(m+1)}{d_\ell(m)}
>\frac{4m^2+7m+\ell+3}{2(m+1)(m-\ell+1)},\quad 1\leq \ell \leq m-1.
\end{align}

In order to establish a sufficiently sharp lower bound for $d_\ell(m)^2/(d_{\ell-1}(m)d_{\ell+1}(m))$, we need the sharp lower bound for $d_\ell(m+1)/d_\ell(m)$ given by Theorem \ref{thm:lb-ratdlm1m}.
To prove this result, we need the following inequality.
\begin{lemma}\label{lem:L>A2/R}
For $m\geq 2$ and $1\leq \ell \leq m-1$, we have
\begin{align}\label{ineq:L>A2/R}
L(m,\ell)
>\frac{(4m+3)(4m+5)(m+\ell+1)}{4(m+1)(m+2)(m-\ell+2)R(m,\ell)},
\end{align}
where
\begin{align*}
R(m,\ell)
=&\ \frac{4m^2+9m-2\ell^2+5}{2(m+2)(m-\ell+2)}
 -\frac{\ell\sqrt{4\ell^4+8\ell^2m+13\ell^2+m+1}}{2(m+2)(m-\ell+2)\sqrt{m+\ell^2+1}}.
\end{align*}
\end{lemma}

\begin{proof}
Fix $m\geq 2$.
We first show that $R(m,\ell)>0$ for $1\leq \ell \leq m-1$. Observe that
\begin{align*}
R(m,\ell)
=&\ \frac{(4m^2+9m-2\ell^2+5)\sqrt{m+\ell^2+1}
  -\ell\sqrt{4\ell^4+8\ell^2m+13\ell^2+m+1}}
  {2(m+2)(m-\ell+2)\sqrt{m+\ell^2+1}}.
\end{align*}
It is easy to check that
\begin{align*}
&\ (4m^2+9m-2\ell^2+5)^2(m+\ell^2+1)-\ell^2(4\ell^4+8\ell^2m+13\ell^2+m+1)\\
=&\ (m+\ell+1)(m-\ell+1)(16\ell^2m^2+40\ell^2m+16m^3+29\ell^2+56m^2+65m+25)>0,
\end{align*}
for $1\leq \ell \leq m-1$, which leads to the fact that $R(m,\ell)>0$ for $1\leq \ell \leq m-1$. Therefore, \eqref{ineq:L>A2/R} is equivalent to
\begin{align}\label{ineq:L>A2/Re}
L(m,\ell)R(m,\ell)
>\frac{(4m+3)(4m+5)(m+\ell+1)}{4(m+1)(m+2)(m-\ell+2)}.
\end{align}

Denote by
\begin{align*}
&U=\sqrt{(m+\ell^2)(4\ell^4+8\ell^2m+5\ell^2+m)},\\
&V=\sqrt{(m+\ell^2+1)(4\ell^4+8\ell^2m+13\ell^2+m+1)}.
\end{align*}
It is easy to verify that
\begin{align*}
&\ L(m,\ell)R(m,\ell)-\frac{(4m+3)(4m+5)(m+\ell+1)}{4(m+1)(m+2)(m-\ell+2)}\\
=&\ \frac{\ell(C_1 U-C_2 V-\ell UV+C_3)}
{4(m+1)(m-\ell+1)(m+\ell^2)(m+2)(m-\ell+2)(m+\ell^2+1)},
\end{align*}
where
\begin{align*}
C_1=&\ (m+\ell^2+1)(4m^2+9m-2\ell^2+5),\\
C_2=&\ (m+\ell^2)(4m^2+7m-2\ell^2+3),\\
C_3=&\ \ell(4\ell^2-1)(m+\ell^2)(m+\ell^2+1).
\end{align*}

It is sufficient to prove
\begin{align}\label{ineq:C1PQ}
C_1 U+C_3>C_2 V+\ell UV, \qquad 1\leq \ell \leq m-1.
\end{align}
Clearly, $C_1,C_2,C_3>0$, and hence $C_1 U+C_3>0,C_2 V+\ell UV>0$ for $1\leq \ell \leq m-1$. So, \eqref{ineq:C1PQ} holds if and only if $(C_1 U+C_3)^2>(C_2 V+\ell UV)^2$ holds. Direct computation gives that
$$
(C_1 U+C_3)^2-(C_2 V+\ell UV)^2=C_4-C_5 U,
$$
where
\begin{align*}
C_4
 =&\ 4(m+\ell^2)(m+\ell^2+1)(m+\ell+1)(m-\ell+1)
 (16\ell^6m+28\ell^6+32\ell^4m^2+76\ell^4m\\
 &\ +32\ell^2m^3+33\ell^4+104\ell^2m^2+102\ell^2m+4m^3+29\ell^2+8m^2+4m),\\
C_5=&\ 8\ell(4\ell^2m+7\ell^2+2m+2)(m+\ell+1)(m-\ell+1)(m+\ell^2)(m+\ell^2+1).
\end{align*}
Moreover, we have
\begin{align*}
C_4^2-C_5^2 U^2=&\ 16(m+\ell^2)^2(m+\ell^2+1)^2(m+\ell+1)^2(m-\ell+1)^2 C_6,
\end{align*}
where
\begin{align*}
C_6
=&\ 256\ell^{10}m^3+960\ell^{10}m^2+1536\ell^8m^4+1136\ell^{10}m
   +7040\ell^8m^3+2048\ell^6m^5+420\ell^{10}\\
&\ +11568\ell^8m^2+11072\ell^6m^4+1024\ell^4m^6+8128\ell^8m
   +22720\ell^6m^3+6912\ell^4m^5\\
&\ +2089\ell^8+22188\ell^6m^2+18272\ell^4m^4+256\ell^2m^6+10340\ell^6m
   +24280\ell^4m^3\\
&\ +1344\ell^2m^5+1834\ell^6+17140\ell^4m^2+2720\ell^2m^4+16m^6+6084\ell^4m
   +2664\ell^2m^3\\
&\ +64m^5+841\ell^4+1264\ell^2m^2+96m^4+232\ell^2m+64m^3+16m^2.
\end{align*}
For $1\leq \ell \leq m-1$, it is clear that $C_4,C_5,C_6>0$. Then $C_4^2-C_5^2 U^2>0$ yields that $C_4-C_5 U>0$, which leads to \eqref{ineq:C1PQ}.
This completes the proof.
\end{proof}

We are now in a position to prove Theorem \ref{thm:lb-ratdlm1m}.

\begin{proof}[Proof of Theorem \ref{thm:lb-ratdlm1m}]
It is easy to verify that \eqref{eq:dlmlb0m} is right. We proceed to prove \eqref{ineq:lb-dlm1m} by mathematical induction on $m$. For $m=2$ and $\ell=1$, we have $d_1(3)/d_1(2)=43/15>17/6=L(2,1)$. Assume that \eqref{ineq:lb-dlm1m} is true, that is, for $1\leq \ell \leq m-1$,
\begin{align}\label{ineq:lb-dlm1mL}
d_\ell(m+1)>L(m,\ell)d_\ell(m).
\end{align}
It suffices to show that for $1\leq \ell \leq m$,
\begin{align}\label{ineq:lb-dlm2m}
d_\ell(m+2)>L(m+1,\ell)d_\ell(m+1).
\end{align}

We first prove \eqref{ineq:lb-dlm2m} for $1\leq \ell \leq m-1$.
By the recurrence relation \eqref{eq:rec-dlm-m}, we have
\begin{align}\label{eq:rec-dlm-m2}
d_\ell(m+2)
=&\ \frac{8m^2+24m-4\ell^2+19}{2(m+2)(m-\ell+2)}d_\ell(m+1) -\frac{(4m+3)(4m+5)(m+\ell+1)}{4(m+1)(m+2)(m-\ell+2)}d_\ell(m).
\end{align}
Applying \eqref{eq:rec-dlm-m2}, the inequality \eqref{ineq:lb-dlm2m} can be restated in the following form:
\begin{align}\label{ineq:lb-dlm1mLr}
R(m,\ell)d_\ell(m+1)
 >\frac{(4m+3)(4m+5)(m+\ell+1)}{4(m+1)(m+2)(m-\ell+2)}d_\ell(m),
\end{align}
where $R(m,\ell)$ is defined in Lemma \ref{lem:L>A2/R}.
By the proof of Lemma \ref{lem:L>A2/R}, $R(m,\ell)>0$ for $1\leq \ell \leq m-1$. Hence, the relation \eqref{ineq:lb-dlm1mLr} is equivalent to the inequality
\begin{align*}
\frac{d_\ell(m+1)}{d_\ell(m)}
>\frac{(4m+3)(4m+5)(m+\ell+1)}{4(m+1)(m+2)(m-\ell+2)R(m,\ell)},
\end{align*}
which can be obtained by \eqref{ineq:lb-dlm1mL} and Lemma \ref{lem:L>A2/R}.

It remains to prove \eqref{ineq:lb-dlm2m} for $\ell=m$. That is,
\begin{align}\label{ineq:lb-dl=mm}
\frac{d_m(m+2)}{d_m(m+1)}>L(m+1,m).
\end{align}
A direct computation gives that
\begin{align*}
 \frac{d_m(m+2)}{d_m(m+1)}
=&\ \frac{(m+1)(4m^2+18m+21)}{2(m+2)(2m+3)},\\
 L(m+1,m)
=&\ \frac{2m^2+15m+14}{4(m+2)}
 +\frac{m\sqrt{4m^4+8m^3+13m^2+m+1}}{4(m+2)\sqrt{m^2+m+1}}.
\end{align*}
Notice that
\begin{align*}
 \frac{d_m(m+2)}{d_m(m+1)}-L(m+1,m)
=&\ \frac{m(4m^4+12m^3+17m^2-(2m+3)W+13m+5)}{4(2m+3)(m+2)(m^2+m+1)},
\end{align*}
where
$$
W=\sqrt{(m^2+m+1)(4m^4+8m^3+13m^2+m+1)}.
$$
It follows that \eqref{ineq:lb-dl=mm} holds if and only if
\begin{align*}
4m^4+12m^3+17m^2+13m+5>(2m+3)W.
\end{align*}
In view of $(4m^4+12m^3+17m^2+13m+5)^2-(2m+3)^2W^2=4(m^2+m+1)(4m^3+19m^2+21m+4)>0$ for $m\geq 2$, we arrive at \eqref{ineq:lb-dl=mm}.
This completes the proof.
\end{proof}

\section{A sharper lower bound for $d_\ell(m)^2/(d_{\ell-1}(m)d_{\ell+1}(m))$}
\label{Sec:3}

This section is devoted to establishing a sufficiently sharp lower bound for $d_\ell(m)^2/(d_{\ell-1}(m)d_{\ell+1}(m))$, which can be used in our proof of Theorem \ref{thm:ht-B-M-ell-sharp}. The desired lower bound is as follows.

\begin{theorem}\label{thm:sharper-bd}
For each $m\geq 2$ and $1\leq \ell \leq m-1$, we have
\begin{align}\label{eq:lb-mell1}
\frac{d_\ell(m)^2}{d_{\ell-1}(m)d_{\ell+1}(m)}
>\frac{(m-\ell+1)(\ell+1)(m+\ell^2)}{(m-\ell)\ell(m+\ell^2+1)}.
\end{align}
\end{theorem}

Chen and Gu had derived the following bounds for $d_\ell(m)^2/(d_{\ell-1}(m)d_{\ell+1}(m))$ \cite[Theorems 1.1, 1.2]{Chen-Gu2009} while studying the reverse ultra log-concavity of the Boros-Moll polynomials.
\begin{theorem}[Chen-Gu]\label{thm:Chen-Gu-1}
For $m\geq 2$ and $1\leq \ell \leq m-1$, there holds
\begin{align}\label{eq:ub-mell2}
\frac{d_\ell(m)^2}{d_{\ell-1}(m)d_{\ell+1}(m)}
<\frac{(m-\ell+1)(\ell+1)}{(m-\ell)\ell}.
\end{align}
\end{theorem}

\begin{theorem}[Chen-Gu]\label{thm:Chen-Gu-2}
For $m\geq 2$ and $1\leq \ell \leq m-1$, there holds
\begin{align}\label{eq:lb-mell2}
\frac{d_\ell(m)^2}{d_{\ell-1}(m)d_{\ell+1}(m)}
>\frac{(m-\ell+1)(\ell+1)(m+\ell)}{(m-\ell)\ell(m+\ell+1)}.
\end{align}
\end{theorem}

Theorem \ref{thm:Chen-Gu-1} is equivalent to the reverse ultra log-concavity of $P_m(x)$. Theorem \ref{thm:Chen-Gu-2} gives an inequality which is stronger than log-concavity of the sequence $\{d_\ell(m)\}_{\ell=0}^m$, that is, \eqref{eq:lb-mell2} leads to $d_\ell(m)^2/(d_{\ell-1}(m)d_{\ell+1}(m))>(i+1)/i$, and hence implies the log-concavity of the sequence $\{\ell! d_\ell(m)\}_{\ell=0}^m$ for $m\geq 2$ \cite[Corollary 4.1]{Chen-Gu2009}.
These two bounds are very sharp. In the asymptotic sense, they suggest that $P_m(x)$ are just on the borderline between ultra log-concavity and reverse ultra log-concavity. See \cite{Chen-Gu2009} for more details.

In view of the two bounds given in \eqref{eq:lb-mell1} and \eqref{eq:lb-mell2}, it is clear that Theorem \ref{thm:sharper-bd} also implies the log-concavity of the sequence $\{\ell! d_\ell(m)\}_{\ell=0}^m$ for $m\geq 2$.

As will be seen in Section \ref{Sec:ht-ell}, the bounds given in Theorems \ref{thm:sharper-bd} and \ref{thm:Chen-Gu-1} behave perfectly in our proof of Theorem \ref{thm:ht-B-M-ell-sharp}.
We proceed to prove Theorem \ref{thm:sharper-bd}.

\begin{proof}[Proof of Theorem \ref{thm:sharper-bd}]
Applying the recurrence relations \eqref{eq:reclm1} and \eqref{eq:reclm2}, the inequality \eqref{eq:lb-mell1} can be restated as
\begin{align*}
A\left(\frac{d_\ell(m+1)}{d_\ell(m)}\right)^2
+B \left(\frac{d_\ell(m+1)}{d_\ell(m)}\right)
+C>0,
\end{align*}
where
\begin{align*}
A=&\ 4(m+1)^2(m-\ell+1)^2(m+\ell^2),\\
B=&\ -4(m+1)(m-\ell+1)(m+\ell^2)(4m^2+7m-2\ell^2+3),\\
C=&\ 16\ell^2m^4-16\ell^4m^2+40\ell^2m^3-32\ell^4m+16m^5+45\ell^2m^2-17\ell^4\\
&\ +56m^4+29\ell^2m+73m^3+9\ell^2+42m^2+9m.
\end{align*}
The discriminant of the above quadratic function in $d_\ell(m+1)/d_\ell(m)$ is
\begin{align*}
\Delta=16\ell^2(m+1)^2(m-\ell+1)^2(m+\ell^2)(4\ell^4+8\ell^2m+5\ell^2+m)>0,
\end{align*}
for $1\leq \ell \leq m-1$.
So, the above quadratic function has two real zeros,
\begin{align*}
&x_1=\frac{4m^2+7m-2\ell^2+3}{2(m+1)(m-\ell+1)}
  -\frac{\ell\sqrt{4\ell^4+8\ell^2m+5\ell^2+m}}{2(m+1)(m-\ell+1)\sqrt{m+\ell^2}},\\
&x_2=\frac{4m^2+7m-2\ell^2+3}{2(m+1)(m-\ell+1)}
  +\frac{\ell\sqrt{4\ell^4+8\ell^2m+5\ell^2+m}}{2(m+1)(m-\ell+1)\sqrt{m+\ell^2}}.
\end{align*}
It remains to show that for $m\geq 2$ and $1\leq \ell \leq m-1$
\begin{align*}
\frac{d_\ell(m+1)}{d_\ell(m)}>x_2,
\end{align*}
which is proved in Theorem \ref{thm:lb-ratdlm1m}, since $x_2=L(m,\ell)$.
This completes the proof.
\end{proof}

Theorem \ref{thm:sharper-bd} gives a sufficiently sharp lower bound for our proof of Theorem \ref{thm:ht-B-M-ell-sharp}. By using the same method, we obtain a sharper bound, which may be available for some deep results on Boros-Moll sequence. The proof is similar to that for Theorem \ref{thm:sharper-bd}, and hence is omitted here.
\begin{theorem}
For each $m\geq 2$ and $1\leq \ell \leq m-1$, we have
\begin{align}\label{eq:lb-mell2l}
\frac{d_\ell(m)^2}{d_{\ell-1}(m)d_{\ell+1}(m)}
>\frac{(m-\ell+1)(\ell+1)(m+\ell+\ell^2)}{(m-\ell)\ell(m+\ell+\ell^2+1)}.
\end{align}
\end{theorem}

\section{Proof of Theorem \ref{thm:ht-B-M-ell-sharp}}\label{Sec:ht-ell}

The objective of this section is to show a simple proof of Theorem \ref{thm:ht-B-M-ell-sharp}, the strict higher order Tur\'{a}n inequalities for the Boros-Moll sequences $\{d_\ell(m)\}_{\ell=0}^m$.
To this end, we need a sufficient condition given by Hou and Li \cite[Theorem 5.2]{Hou-Li2021}.

\begin{theorem}[Hou-Li]\label{thm:Hou-Li}
Let $\{a_n\}_{n\geq 0}$ be a real sequence with positive numbers. Let
$$
d(x,y)=4(1-x)(1-y)-(1-xy)^2.
$$
If there exist an integer $N$, and two functions $g(n)$ and $h(n)$ such that for all $n\geq N$,
\begin{itemize}
\item[$(i)$] $g(n)<a_{n-1}a_{n+1}/a_n^2<h(n)$;
\item[$(ii)$] $d(g(n),g(n+1))>0$, $d(g(n),h(n+1))>0$, $d(h(n),g(n+1))>0$, $d(h(n),h(n+1))>0$,
\end{itemize}
then $\{a_n\}_{n\geq N-1}$\footnote{Note that it was showed $\{a_n\}_{n\geq N}$ in the original literature. It is easy to see that the result is also true for $\{a_n\}_{n\geq N-1}$ for $N\geq 1$.} satisfies the higher order Tur\'{a}n inequalities.
\end{theorem}

\begin{remark}
It is easy to check that the sequence described in Theorem \ref{thm:Hou-Li} satisfies the higher order Tur\'{a}n inequalities strictly since all the inequalities in conditions $(i)$ and $(ii)$ are strict.
\end{remark}

We are now ready to prove Theorem \ref{thm:ht-B-M-ell-sharp}.

\begin{proof}[Proof of Theorem \ref{thm:ht-B-M-ell-sharp}]
Fix $m\geq 2$. For $1\leq \ell \leq m-1$, by Theorems \ref{thm:sharper-bd} and \ref{thm:Chen-Gu-1}, we have
\begin{align}\label{ineq:two-bds}
 \frac{(m-\ell)\ell}{(m-\ell+1)(\ell+1)}
<\frac{d_{\ell-1}(m)d_{\ell+1}(m)}{d_\ell(m)^2}
<\frac{(m-\ell)\ell(m+\ell^2+1)}{(m-\ell+1)(\ell+1)(m+\ell^2)}.
\end{align}
In order to use Theorem \ref{thm:Hou-Li}, for $1\leq n \leq m-1$, set $a_n=d_n(m)$ and
$$
g(n)=\frac{(m-n)n}{(m-n+1)(n+1)},\qquad
h(n)=\frac{(m-n)n(m+n^2+1)}{(m-n+1)(n+1)(m+n^2)}.
$$
Let $N=1$. Then by \eqref{ineq:two-bds}, the condition $(i)$ in Theorem \ref{thm:Hou-Li} is satisfied for $N\leq n \leq m-1$.

It suffices to verify the four inequalities in $(ii)$ of Theorem \ref{thm:Hou-Li}.
By a direct computation, we have
\begin{align*}
 d(g(n),g(n+1))
&=\frac{4(m+1)^2(m+2)}{(m-n)(m-n+1)^2(n+1)(n+2)^2}>0
\end{align*}
for $1\leq n\leq m-1$.
Similarly, we obtain that
\begin{align*}
 d(g(n),h(n+1))
&=\frac{F}{(m-n)(m-n+1)^2(n+1)(n+2)^2(n^2+2n+m+1)^2},
\end{align*}
where
\begin{align*}
 F=
&\ 4m^3n^4+8m^4n^2+7m^3n^3+mn^4(31m-7n)+n^6+4m^5+8m^4n+43m^3n^2+93m^2n^3\\
&\ +n^4(17m-n)+16m^4+60m^3n+162m^2n^2+115mn^3+3n^4+40m^3+156m^2n\\
&\ +203mn^2+41n^3+64m^2+164mn+80n^2+52m+60n+16>0.
\end{align*}
Clearly, $d(g(n),h(n+1))>0$ for $1\leq n\leq m-1$.
Moreover,
\begin{align*}
 d(h(n),g(n+1))
&=\frac{G}{(m-n)(m-n+1)^2(n^2+m)^2(n+1)(n+2)^2},
\end{align*}
where
\begin{align*}
 G=
&\ 4m^3n^4+m^3n^2(8m-5n)+19m^2n^4+mn^5+n^6+4m^4(m-n)+m^2n^2(31m-7n)\\
&\ +25mn^4+7n^5+16m^3(m-n)+mn^2(54m-13n)+23n^4+m^2(20m-12n)\\
&\ +27mn^2+n^3+8m^2.
\end{align*}
Observe that $G>0$ and hence $d(h(n),g(n+1))>0$ for  $1\leq n\leq m-1$.
Finally, we have
\begin{align*}
 d(h(n),h(n+1))
=\frac{4H}{(m-n)(m-n+1)^2(n^2+m)^2(n^2+2n+m+1)^2(n+1)(n+2)^2},
\end{align*}
where
\begin{align*}
H=&\,
(m+3)(m+2)^2n^8+(m+8)(m+3)(m+2)n^7+(4m^4+24m^3+47m^2\\
&\, +62m+73)n^6+(3m^4+38m^3+101m^2+67m+51)n^5+(6m^5+31m^4\\
&\, +60m^3+117m^2+65m+15)n^4+(3m^5+38m^4+68m^3+83m^2\\
&\, +43m+1)n^3+(m^2+m)(4m^4+14m^3+31m^2+22m+13)n^2\\
&\, +m^2(m+9)(m+1)^3n+m^2(m^2+m+4)(m+1)^3>0.
\end{align*}
It is clear that $d(h(n),h(n+1))>0$ for  $1\leq n\leq m-1$.
So the four inequalities in $(ii)$ of Theorem \ref{thm:Hou-Li} hold for each $m\geq 2$ and $1\leq n\leq m-1$.

Thus, for each $m\geq 3$, we have that $\{a_n\}_{n=0}^m$, i.e., $\{d_\ell(m)\}_{\ell=0}^m$, satisfies the higher order Tur\'{a}n inequalities strictly.
\end{proof}

\section*{Acknowledgments}
The author would like to thank Qing-Hu Hou and Larry X. W. Wang for helpful discussion.

\end{document}